\documentclass[a4paper,12pt]{article}
\usepackage{amsmath,amsthm,amsfonts,amssymb,color}
\usepackage[matrix,arrow]{xy}
\usepackage{graphicx}
\usepackage[UKenglish]{babel}
\usepackage{multicol} 
\usepackage{appendix}
\usepackage{epstopdf}
\usepackage{graphicx}
\usepackage{epstopdf}
\usepackage{multicol}
\usepackage{caption}
\usepackage{subcaption}

\usepackage{multirow} 
\usepackage{float}
\usepackage{graphicx}
\usepackage{longtable}  
\usepackage{graphicx,epstopdf}

\usepackage{amsmath,amssymb,amsthm}
\usepackage{enumerate}
\usepackage{amsfonts}
\usepackage{graphics,graphicx}
\usepackage{verbatim}
\usepackage{latexsym,makeidx}
\usepackage{amsmath,amscd}
\usepackage{appendix}
\usepackage[latin1]{inputenc}

\newenvironment{keywords}{
  \vspace{2mm}
  \noindent
  \keywordsname: 
  \itshape\small
}

\newenvironment{mathsubclass}{
  \small
  \noindent
  \mathsubclassname: 
}

 \def\keywordsname{\textbf{Keywords}}
  \def\mathsubclassname{\textbf{2010 AMS Subject Classification}}

\newtheorem{teo}{Theorem}[section]

\newtheorem{coro}{Corollary}[section]

\newtheorem{obs}{Remark}[section]

\newcommand\be{\begin{equation}}
\newcommand\ee{\end{equation}}

\begin{document}

\title{One-phase Stefan-like problems with a  latent heat depending on
 the position and velocity of the free boundary,  and with  Neumann or Robin boundary conditions at the fixed face}

\author{
Julieta Bollati$^{1}$,  Domingo A. Tarzia $^{1}$\\ \\
\small {{$^1$} Depto. Matem\'atica - CONICET, FCE, Univ. Austral, Paraguay 1950} \\  
\small {S2000FZF Rosario, Argentina.}\\
\small{Email: JBollati@austral.edu.ar; DTarzia@austral.edu.ar.} 
}
\date{}

\maketitle

\abstract{In this paper, a one-phase Stefan-type problem for a semi-infinite material  which has as its main feature a variable latent heat that depends on the power of the position and the velocity of the moving boundary is studied. Exact solutions of similarity type are obtained for the cases when  Neumann or Robin boundary conditions are imposed at the fixed face. Required relationships between data are presented in order that these problems become equivalent to the problem where a Dirichlet condition at the fixed face is considered. Moreover, in the case where a Robin condition is prescribed, the limit behaviour is studied when the heat transfer coefficient at the fixed face goes to infinity. 
}

\begin{keywords}
Stefan problem, Threshold gradient, Variable latent heat,  One-dimensional consolidation, Explicit solution, Similarity solution.
\end{keywords}
\begin{mathsubclass}
35C05, 35R35, 80A22.
\end{mathsubclass}

\section{Introduction}

Stefan-like problems  have attracted growing attention in the last decades due to the fact that they  arise in many significant areas of engineering, geoscience and industry \cite{AlSo}-\cite{Ta4}.   The classical Stefan problem  describes the process of a material undergoing a phase change. Finding a solution to this problem consists in solving the heat-conduction equation in an unknown region  which has also to be determined, imposing an initial condition, boundary conditions and the Stefan condition at the moving interface. For an account of the theory we refer the reader to \cite{Ta2}.

In the classical Stefan problem the latent heat is assumed to be constant. In this paper, we are going to consider a variable one. This assumption is motivated by the fact that it becomes meaningful in the study of the shoreline movement  in a sedimentary basis \cite{VSP},  in the one-dimensional consolidation with threshold gradient \cite{ZBL}, in the artificial ground-freezing technique \cite{ZSZ},  in nanoparticle melting \cite{RiMy16}, among others ( \cite{Do14}-\cite{ZhXi} )

Many papers deal with a latent heat that depends on the position of the free boundary (size-dependent latent heat). In  \cite{Pr}, a Stefan problem with a latent heat given as a function of the position of the interface $L=\varphi(s(t))$ has been considered. This hypothesis corresponds to the practical case when the influence of phenomena such as surface tension, pressure gradients and non-homogeneity of materials are taken into account. 
In \cite{VSP},  the shoreline movement in a sedimentary basin was studied, from where arises a one-phase Stefan problem with a latent heat that increases
linearly with distance from the origin i.e. $L=\gamma s(t)$ (with $\gamma$  a given constant). The generalization to the two-phase problem was done in \cite{SaTa}.
Also, in \cite{ZWB} a latent heat  defined as   a power function of the position, i.e. $L=\gamma s^{n}(t)$ (with $\gamma$ a given constant and $n$ an arbitrary non-negative integer) was considered. The extension to a non-integer exponent was done in \cite{ZhXi} for a flux and temperature boundary conditions, while the two-phase case was presented in \cite{ZSZ}. In \cite{BoTa-1},  a convective (Robin) condition  was 
imposed for the one-phase case while for the  two-phase case the analysis was done in \cite{BoTa-2}.

In \cite{ZBL}, a one-dimensional consolidation problem with a threshold gradient was studied. This problem can be reduced to a one-phase Stefan problem where the latent heat can be expressed as $L=\dfrac{\gamma}{\dot{s}(t)}$. That is to say a rate-dependent  latent heat. It must be noticed that the case considered in \cite{ZBL} is not properly a Stefan problem because the velocity of the moving boundary disappears, and it has to be treated as a free boundary problem with implicit conditions \cite{Fa}, \cite{Sc}.

Recently, in \cite{BoTa}  it was  defined a generalized one-phase Stefan-like problem  for a semi-infinite material $x>0$ with a latent heat given by $L=\gamma s^{\beta}(t)\dot{s}^{\delta}(t)$ (with $\gamma$ a given constant and $\beta$ and $\delta$ arbitrary real constants), i.e., a latent heat depending on the position and velocity of the moving boundary, taking a Dirichlet boundary condition. This paper  intends to complete this model, by considering two new boundary conditions (Neumann and Robin conditions) at the fixed face $x=0$. 

In Section 2 we present a problem $(P)$ with a variable latent heat and a generalized boundary condition at the fixed face. We will obtain its exact solution, following the methodology given in \cite{ZBL}, \cite{ZhXi}, and \cite{BoTa}, obtaining as immediate consequence the similarity solutions to two different  problems: one with Neumann condition at $x=0$ and the other with a Robin one.  Special cases will be treated in order to recover solutions recently reported in literature.    Moreover, in Section 3, the equivalence between these  problems and the problem with a Dirichlet condition considered in \cite{BoTa} will be proved under certain relationships between data. For the problem with a Robin boundary condition at the fixed face,  the limit behaviour when the heat transfer  coefficient  goes to infinity will  be also  analysed in Section 4. This analysis will allow us to show that the Robin condition constitutes a generalization of the Dirichlet one, as  happens in classical heat transfer problems \cite{Ta17}.
Also, in Section 5, we will provide some plots and table of values in order to track the position of the free front and to show how the latent heat changes in time.

\section{Formulation of the problems and exact solution}

\subsection{Statement of the problems}
In this paper, the exact solution of two different free boundary problems are obtained. They will be defined as particular cases of the following problem $(P)$ that consists in  finding the function $u=u(x,t)$ and the moving boundary $x=s(t)$ such that:
\begin{align}
&\frac{\partial u}{\partial t}(x,t)=a^2 \frac{\partial^2 u}{\partial x^2}(x,t), & 0<x<s(t), \label{EcCalor} \\
&u(s(t),t)=0, & t>0, \label{TempCambioFase}\\ 
&-k\frac{\partial u}{\partial x}(s(t),t)= L(s(t),\dot{s}(t))\dot{s}(t), &t>0, \label{CondStefan}\\
&s(0)=0, &\label{CondInicialFrontera}\\
& k \frac{\partial u}{\partial x}(0,t) =\frac{h_0}{\sqrt{t}}\left[\lambda u(0,t)-u_{\infty}t^{\tfrac{\alpha}{2}} \right],&\qquad t>0,  \label{CondConvFronteraG}
\end{align}
where $\alpha$, $\lambda$, $h_0$ , $a^2$  (diffusivity) and  $k$ (conductivity) are non-negative constants.

The problem defined by specifying $\lambda=0$ in $(P)$, will be referred to as \textbf{problem} ($P_N$). In this case, condition (\ref{CondConvFronteraG}) corresponds to the Neumann boundary  condition:
\begin{equation}
 k \frac{\partial u}{\partial x}(0,t) =-q_0 t^{\tfrac{\alpha-1}{2}},\qquad   t>0 , \quad (q_0>0) \label{CondFlujoFrontera}
\end{equation}
where a time dependent heat flux characterized by $q_0=h_0u_{\infty}>0$ is applied at the fixed face $x=0$. This flux is  proportional to the power $\tfrac{\alpha-1}{2}$ of time.

The problem defined by specifying $\lambda=1$ in $(P)$ will be referred to as \textbf{problem} ($P_R$). In this case, condition (\ref{CondConvFronteraG}) corresponds to the Robin boundary  condition:
\begin{equation}
 k \frac{\partial u}{\partial x}(0,t) =\frac{h_0}{\sqrt{t}}\left[ u(0,t)-u_{\infty}t^{\tfrac{\alpha}{2}} \right],\qquad t>0, \quad (h_0>0) \label{CondConvFrontera}
\end{equation}
where $u_{\infty}$ characterizes the bulk temperature at a large distance from the fixed face $x = 0$ and $h_0>0$ characterizes the heat transfer at the fixed face.

Comparing these problems with respect to the classical Stefan problem, the new feature to be observed is that the condition (\ref{CondStefan}) at the free interface can be thought of as a generalized Stefan condition where the latent heat term $L(s(t),\dot{s}(t))$ is not constant but rather a function of position and velocity of the moving boundary. Furthermore, in order to obtain a similarity type solution for problems $(P)$, $(P_N)$ and $(P_R)$, $L$ will be specified as:
\begin{equation}
L(s(t),\dot{s}(t))= \gamma s^{\beta}(t)\dot{s}^{\delta}(t),
\end{equation}
with $\gamma$, $\beta$ and $\delta$ non-negative given constants.

\subsection{Similarity-type solutions}

Before finding the similarity type solution to problem $(P)$,  the subsequent analysis will be necessary.

Let us observe that if we use the following similarity transformation presented in \cite{{ZWB}} and \cite{ZhXi}:
\begin{equation}
u(x,t)=t^{\tfrac{\alpha}{2}} \varphi(\eta) \qquad \text{ with }\qquad  \eta=\frac{x}{2a\sqrt{t}},\label{Transformacion}
\end{equation}
then it is obtained that
$$\frac{\partial^2 u}{\partial x^2}(x,t)=t^{\tfrac{\alpha}{2}-1} \varphi''(\eta) \frac{1}{4a^2}\qquad \text{and} \qquad \frac{\partial u}{\partial t}(x,t)=t^{\tfrac{\alpha}{2}-1} \left[ \frac{\alpha}{2}\varphi(\eta)-\eta \varphi''(\eta)\right].$$
Therefore, equation (\ref{EcCalor}) is satisfied, i.e. $\tfrac{\partial u}{\partial t}(x,t)=a^2 \tfrac{\partial^2 u}{\partial x^2}(x,t)$ if and only if
\begin{equation}
\varphi''(\eta)+2\eta \varphi'(\eta)-2\alpha \varphi(\eta)=0. \label{EcKummer}
\end{equation}
This second order ordinary differential equation, known in literature as Kummer's differential equation (see \cite{OLBC}), has a  general solution that is given by
\begin{equation}
\varphi(\eta)=t^{\tfrac{\alpha}{2}}\left[C_1 M\left(-\tfrac{\alpha}{2},\tfrac{1}{2},-\eta^2 \right)+C_2 \eta M\left(-\tfrac{\alpha}{2}+\tfrac{1}{2},\tfrac{3}{2},-\eta^2 \right)  \right], \label{SolKummer}
\end{equation}
with $C_1$ and $C_2$  arbitrarily constants. The function $M(a,b,z)$ is called Kummer's function or confluent hypergeometric function of the first kind and it is defined by the following series
$$M(a,b,z)=\sum\limits_{n=0}^{\infty} \frac{(a)_n}{(b)_n} \frac{z^n}{n!},$$
where $b$ cannot be a non-positive integer, and $(a)_n$ is the Pochhammer symbol defined by
 $$(a)_0=1,\qquad  (a)_n= a (a+1) (a+2)\cdots (a+n-1).$$

The detailed  proof of the fact that the general solution of the Kummer's equation (\ref{EcKummer}) can be written as (\ref{SolKummer}) may be found in \cite{BoTa}.

The main properties of the Kummer's function $M(a,b,z)$ to be used throughout this paper can be found in \cite{OLBC} and they are stated in the following way:
\begin{align}
& M(a,b,0)=1, \label{Kummer0}\\
& M(a,b,z)= e^z M(b-a,b,-z),\label{RelacionExponencial-1} \\
& e^{-z^2}= -2\alpha z^2 M\left(-\tfrac{\alpha}{2}+\tfrac{1}{2},\tfrac{3}{2},-z^2 \right) M\left(-\tfrac{\alpha}{2}+1,\tfrac{3}{2},-z^2 \right) +\nonumber \\
& \qquad+ M\left(-\tfrac{\alpha}{2},\tfrac{1}{2},-z^2 \right) M\left(-\tfrac{\alpha}{2}+\tfrac{1}{2},\tfrac{1}{2},-z^2 \right),\label{RelacionExponencial-2} \\
& \frac{d}{dz} M(a,b,z)=\frac{a}{b} M(a+1,b+1,z), \label{DerivadaKummer-1}\\
& \frac{d}{dz}\left[ z^{b-1} M(a,b,z)\right]= (b-1) z^{b-2} M(a,b-1,z),\label{DerivadaKummer-2}\\
& M(a,b,z)\simeq \frac{e^z z^{a-b}}{\Gamma(a)} \quad \text{ when }\quad  z\to\infty,\label{KummerInf}\\
&M\left(-\dfrac{n}{2},\dfrac{1}{2},-z^2 \right)  =2^{n-1} \Gamma\left(\dfrac{n}{2}+1 \right)\left[i^n erfc(z)+i^nerfc(-z) \right]\label{Prop-3},\quad n\in\mathbb{N},\\
&zM\left(-\dfrac{n}{2}+\dfrac{1}{2},\dfrac{3}{2},-z^2 \right)=2^{n-2}\Gamma\left( \dfrac{n}{2}+\dfrac{1}{2}\right)\left[i^n erfc(-z)-i^n erfc(z) \right] ,\quad n\in\mathbb{N},\label{Prop-4}
\end{align}
where $i^nerfc(\cdot)$ is the repeated integral of the complementary error function defined by
\begingroup
\addtolength{\jot}{0.1em}
\begin{align*}
& i^0 erfc(z)=erfc(z)=1- erf(z), \qquad erf(z)=\dfrac{2}{\sqrt{\pi}}\int_0^z e^{-u^2}du, \\
& i^n erfc(z)=\int\limits_{z}^{+\infty} i^{n-1}erfc(t)dt. 
\end{align*}
\endgroup

Now, we will look for the similarity solution to problem ($P$). In order to make the notation clearer, we will refer to the solution of problem ($P$) as the pair $(u(x,t), s(t))$ that satisfies (\ref{EcCalor})-(\ref{CondConvFronteraG}).

According to the previous analysis, $u$ will satisfy equation (\ref{EcCalor}) if it is written as:
\begin{equation}
u(x,t)=t^{\tfrac{\alpha}{2}}\left[C_{1} M\left(-\tfrac{\alpha}{2},\tfrac{1}{2},-\eta^2 \right)+C_{2} \eta M\left(-\tfrac{\alpha}{2}+\tfrac{1}{2},\tfrac{3}{2},-\eta^2 \right)  \right], \label{FormaGeneralu}
\end{equation}
with the similarity variable given by $\eta=\tfrac{x}{2a\sqrt{t}}$ and where $C_{1}$, $C_{2}$  are constants to be determined  so that $u$ satisfies the rest of the conditions.

Observe that from (\ref{TempCambioFase}) it should be that $\varphi$ defined by the transformation (\ref{Transformacion}) has to satisfy  $\varphi\left(\tfrac{s(t)}{2a\sqrt{t}} \right)=0$ for all $t>0$. Therefore the moving boundary must adopt the following form:
\begin{equation}
s(t)=2\xi a \sqrt{t},\label{FormaGenerals}
\end{equation}
where $\xi$ is a positive dimensionless coefficient  to be determined.

Hence, bearing in mind that $u$ is written as (\ref{FormaGeneralu}) and $s$ as (\ref{FormaGenerals}), finding the solution to problem ($P$) consists in determining the coefficients $C_{1}$, $C_{2}$ and $\xi$.

The generalized boundary condition at the fixed face (\ref{CondConvFronteraG}), and properties (\ref{Kummer0}), (\ref{DerivadaKummer-1})-(\ref{DerivadaKummer-2}) imply that:
\begin{equation}
C_{2}=\dfrac{2a h_0}{k} \left(\lambda C_1-u_\infty \right). \label{C2-previa}
\end{equation}

From condition (\ref{TempCambioFase}), it can be deduced after some computations that:
\begin{equation}
C_{1}=\dfrac{u_\infty \xi M\left( -\tfrac{\alpha}{2}+\tfrac{1}{2},\tfrac{3}{2},-\xi^2\right)}{\tfrac{k}{2ah_0} M\left(-\tfrac{\alpha}{2},\tfrac{1}{2},-\xi^2 \right)+\lambda \xi M\left( -\tfrac{\alpha}{2}+\tfrac{1}{2},\tfrac{3}{2},-\xi^2\right)}.\label{C1}
\end{equation}
 Therefore, replacing $C_1$ in (\ref{C2-previa}), we get
\begin{equation}
C_2=\dfrac{-u_\infty  M\left( -\tfrac{\alpha}{2},\tfrac{1}{2},-\xi^2\right)}{\tfrac{k}{2ah_0} M\left(-\tfrac{\alpha}{2},\tfrac{1}{2},-\xi^2 \right)+\lambda \xi M\left( -\tfrac{\alpha}{2}+\tfrac{1}{2},\tfrac{3}{2},-\xi^2\right)}.\label{C2}
\end{equation}
Then, we have obtained  $C_{1}$ and $C_2$ as  functions of $\xi$.

Finally,  the Stefan-type condition given by (\ref{CondStefan})  will give us an equation for $\xi$.

Applying the derivation formulas (\ref{DerivadaKummer-1})-(\ref{DerivadaKummer-2}) we claim that:
\begin{align*}
\frac{\partial u}{\partial x}(s(t),t)&= \frac{t^{\tfrac{\alpha-1}{2}} u_\infty}{2a}\tfrac{ \left[2\alpha \xi^2 M\left( -\tfrac{\alpha}{2}+\tfrac{1}{2},\tfrac{3}{2},-\xi^2\right) M\left( -\tfrac{\alpha}{2}+1,\tfrac{3}{2},-\xi^2\right)-M\left( -\tfrac{\alpha}{2},\tfrac{1}{2},-\xi^2\right)M\left( -\tfrac{\alpha}{2}+\tfrac{1}{2},\tfrac{1}{2},-\xi^2\right)\right]}{ \left[ \tfrac{k}{2ah_0} M\left(-\tfrac{\alpha}{2},\tfrac{1}{2},-\xi^2 \right)+\lambda \xi M\left( -\tfrac{\alpha}{2}+\tfrac{1}{2},\tfrac{3}{2},-\xi^2\right)\right] }.
\end{align*}
Using the relationships  (\ref{RelacionExponencial-1})-(\ref{RelacionExponencial-2}), the partial derivative of $u$ is reduced to 
\begin{equation}\label{Derivadau}
\frac{\partial u}{\partial x}(s(t),t)= \frac{- t^{\tfrac{\alpha-1}{2}}u_\infty}{2a}\frac{1}{ \left[ \tfrac{k}{2ah_0} M\left(\tfrac{\alpha}{2}+\tfrac{1}{2},\tfrac{1}{2},\xi^2 \right)+\lambda \xi M\left( \tfrac{\alpha}{2}+1,\tfrac{3}{2},\xi^2\right)\right]}.
\end{equation}
Replacing (\ref{Derivadau}) in (\ref{CondStefan}) yields to the following equality:
$$ \frac{ k u_\infty t^{\tfrac{\alpha-1}{2}}}{2a \left[ \tfrac{k}{2ah_0} M\left(\tfrac{\alpha}{2}+\tfrac{1}{2},\tfrac{1}{2},\xi^2 \right)+\lambda \xi M\left( \tfrac{\alpha}{2}+1,\tfrac{3}{2},\xi^2\right)\right]}= \gamma 2^{\beta }a^{\beta+\delta+1} t^{\tfrac{\beta-\delta-1}{2}} \xi^{\beta+\delta+1}, $$
which makes sense if and only if $\tfrac{\alpha-1}{2}=\tfrac{\beta-\delta-1}{2}$, due to the fact that nor $\gamma$, $\xi$, or $a$ depends on time. Thus, the similarity solution for problem ($P$) will exist if and only if
\begin{equation}
\alpha=\beta-\delta\geq 0,\label{Restriccion-1}
\end{equation}
and if $\xi$ is a  positive solution of the following equation:
 \begin{equation}
 \dfrac{k u_\infty}{\gamma 2^{\beta+1} a^{\beta+\delta+2}} f_\lambda(z)=z^{\beta+\delta+1},\qquad z>0, \label{EcXi}
 \end{equation}
 with
 \begin{equation}\label{flambda}
 f_\lambda(z)=\frac{1}{  \tfrac{k}{2ah_0} M\left(\tfrac{\alpha}{2}+\tfrac{1}{2},\tfrac{1}{2},z^2 \right)+\lambda z M\left( \tfrac{\alpha}{2}+1,\tfrac{3}{2},z^2\right)   }.
 \end{equation}
 
 The notation $f_\lambda$ is adopted in order to emphasize the dependence
of the solution to problem $(P)$ on $\lambda$, and therefore to obtain easily the solutions to problems $(P_N)$ and  $(P_R)$.

 The task is now to prove the existence and uniqueness of solution to equation (\ref{EcXi}). 
From the relationships (\ref{Kummer0}), (\ref{DerivadaKummer-1}),  (\ref{DerivadaKummer-2}) and (\ref{KummerInf}), we obtain that $f_\lambda$  satisfies:
\begin{align}
& f_\lambda'(z)= \frac{-\left[ \tfrac{k}{2ah_0} (\alpha+1)z M\left( \tfrac{\alpha}{2}+\tfrac{3}{2},\tfrac{3}{2},z^2\right)+\lambda M\left( \tfrac{\alpha}{2}+1,\tfrac{1}{2},z^2\right)\right]}{\left[ \tfrac{k}{2ah_0} M\left(\tfrac{\alpha}{2}+\tfrac{1}{2},\tfrac{1}{2},z^2 \right)+\lambda z M\left( \tfrac{\alpha}{2}+1,\tfrac{3}{2},z^2\right)    \right]}     <0, \label{dg}\\
& f_\lambda(0^{+})=\dfrac{2a h_0}{k}>0,\\
& f_\lambda(+\infty)=0.
\end{align}

We can deduce that the l.h.s. of equation (\ref{EcXi}) is a strictly decreasing function that goes from $\tfrac{u_\infty h_0}{\gamma 2^{\beta} a^{\beta+\delta+1}}>0$ to $0$ when $z$ increases from 0 to $+\infty$, while the r.h.s. of equation  (\ref{EcXi}), if $\beta+\delta+1>0$, is a strictly increasing function that goes from 0 to $+\infty$.

As a conclusion we obtain that if $\beta+\delta+1>0$, we can ensure that (\ref{EcXi}) has a unique positive solution.

It should be mentioned that due to the restrictions (\ref{Restriccion-1}), i.e. $\alpha=\beta-\delta\geq 0$ and  $\beta+\delta+1>0$ we get that: $\beta\geq \max(\delta,-1-\delta)$.

All the above analysis can be summarized in the following theorem:

\begin{teo}\label{teo: P}
Let $\beta$ and $\delta$ be arbitrary real constants satisfying $\beta\geq \max(\delta,-1-\delta)$. Taking $\alpha=\beta-\delta$, there exists a unique solution $(u,s )$  of a similarity type for problem ($P$), i.e. equations (\ref{EcCalor})-(\ref{CondConvFronteraG}),
 which is given  by (\ref{FormaGeneralu}) and (\ref{FormaGenerals}), where $C_{1}$ and $C_{2}$ are given by the formulas (\ref{C1}) and (\ref{C2}) respectively,  and the dimensionless coefficient $\xi$ is defined as the unique positive solution of the equation (\ref{EcXi}).
\end{teo}


The solutions to problems $(P_N)$ and $(P_R)$ can be obtained as a consequence of Theorem \ref{teo: P}, by fixing $\lambda=0$ or $\lambda=1$, respectively.
As an immediate consequence  we have the following results.

\begin{coro}{\sc\textbf{ (case $\lambda=0$)}}\label{teo: Flujo}
Let $\beta$ and $\delta$ be arbitrary real constants satisfying \mbox{$\beta\geq \max(\delta,-1-\delta)$}. Taking $\alpha=\beta-\delta$, there exists a unique solution $(u_N,s_N )$  of a similarity type for\textbf{ problem ($P_N$)}, i.e. equations (\ref{EcCalor})-(\ref{CondInicialFrontera}), and (\ref{CondFlujoFrontera}) which is given by
\begin{align}
&u_N(x,t)=t^{\tfrac{\alpha}{2}}\left[C_{1N} M\left(-\tfrac{\alpha}{2},\tfrac{1}{2},-\eta^2 \right)+C_{2N} \eta M\left(-\tfrac{\alpha}{2}+\tfrac{1}{2},\tfrac{3}{2},-\eta^2 \right)  \right], \label{FormaGeneralu-flujo}\\
&s_N(t)=2\xi_N a \sqrt{t},\label{FormaGenerals-flujo}
\end{align}
where $\eta=\tfrac{x}{2a\sqrt{t}}$ is the similarity variable. The coefficients $C_{1N}$, $C_{2N}$ are defined by
$$C_{1N}= \frac{2 a q_0 \xi_N}{k} \frac{M\left( -\tfrac{\alpha}{2}+\tfrac{1}{2},\tfrac{3}{2},-\xi_N^2\right)}{M\left( -\tfrac{\alpha}{2},\tfrac{1}{2},-\xi_N^2\right)},\qquad \qquad C_{2N}=\dfrac{-2aq_0}{k}, $$
and $\xi_N$ is the unique positive solution of the equation $\tfrac{ku_\infty}{\gamma 2^{\beta+1}a^{\beta+\delta+2}}f_0(z)=z^{\beta+\delta+1}$, that can be rewritten as 
 \begin{equation}
 \frac{q_0}{\gamma 2^{\beta} a^{\beta+\delta+1}} g(z)=z^{\beta+\delta+1}, \qquad z>0, \label{EcxiFlujo}
 \end{equation}
 with
 \begin{equation}\label{f0}
 g(z)=\frac{1}{M\left( \tfrac{\alpha}{2}+\tfrac{1}{2},\tfrac{1}{2},z^2\right)}.
 \end{equation}
\end{coro}

\begin{coro}{\sc\textbf{(case $\lambda=1$)} }\label{teo: Conv}
Let $\beta$ and $\delta$ be arbitrary real constants satisfying \mbox{$\beta\geq \max(\delta,-1-\delta)$}. Taking $\alpha=\beta-\delta$, there exists a unique solution $(u_R,s_R )$  of a similarity type for\textbf{ problem ($P_R$)}, i.e. equations (\ref{EcCalor})-(\ref{CondInicialFrontera}) and (\ref{CondConvFrontera}),
 which is given  by 
 \begin{align}
&u_R(x,t)=t^{\tfrac{\alpha}{2}}\left[C_{1R} M\left(-\tfrac{\alpha}{2},\tfrac{1}{2},-\eta^2 \right)+C_{2R} \eta M\left(-\tfrac{\alpha}{2}+\tfrac{1}{2},\tfrac{3}{2},-\eta^2 \right)  \right], \label{FormaGeneralu-Conv}\\
&s_R(t)=2\xi_R a \sqrt{t},\label{FormaGenerals-Conv}
\end{align}
where $\eta=\tfrac{x}{2a\sqrt{t}}$ is the similarity variable. The coefficients $C_{1R}$, $C_{2R}$ are defined by
$$C_{1R}=\tfrac{u_\infty \xi_R M\left( -\tfrac{\alpha}{2}+\tfrac{1}{2},\tfrac{3}{2},-\xi_R^2\right)}{\tfrac{k}{2ah_0} M\left(-\tfrac{\alpha}{2},\tfrac{1}{2},-\xi_R^2 \right)+ \xi_R M\left( -\tfrac{\alpha}{2}+\tfrac{1}{2},\tfrac{3}{2},-\xi_R^2\right)},\qquad C_{2R}=\tfrac{-u_\infty  M\left( -\tfrac{\alpha}{2},\tfrac{1}{2},-\xi_R^2\right)}{\tfrac{k}{2ah_0} M\left(-\tfrac{\alpha}{2},\tfrac{1}{2},-\xi_R^2 \right)+ \xi_R M\left( -\tfrac{\alpha}{2}+\tfrac{1}{2},\tfrac{3}{2},-\xi_R^2\right)}, $$
and $\xi_R$ is the unique positive solution of the following equation
 \begin{equation}
\frac{ku_\infty}{\gamma 2^{\beta+1}a^{\beta+\delta+2}}f_1(z)=z^{\beta+\delta+1},\qquad z>0,\label{EcXiConv}
 \end{equation}
 with $f_1$ defined by replacing $\lambda=1$ in $f_\lambda$ given in (\ref{flambda}):
 \begin{equation}\label{f1}
f_{1}(z)=\frac{1}{\left[\tfrac{k}{2ah_0}M\left( \tfrac{\alpha}{2}+\tfrac{1}{2},\tfrac{1}{2}, z^2 \right)+z M\left( \tfrac{\alpha}{2}+1,\tfrac{3}{2},z^2\right)\right]}.
 \end{equation}
\end{coro}

Specifying different values for $\beta$ and $\delta$ in the above results, several solutions reported in literature can be recovered as a corollary. For instance:

\begin{coro}
The solution to the classical Stefan problem with a Neumann boundary condition at the fixed face can be recovered from Theorem \ref{teo: P}  by taking $\lambda=0$, $\beta=\delta=0$.
\end{coro}

\noindent Taking $\beta=\delta=0$ and thus $\alpha=0$, the latent heat $L=\gamma$ is assumed to be constant like in the classical Stefan problem. In such case, fixing $\lambda=0$, the flux boundary condition  is given by $k\frac{\partial u}{\partial x}(0,t)=-\frac{q_0}{\sqrt{t}}$. Moreover, the property  $M\left(\frac{1}{2},\frac{1}{2},z \right)=e^z$, allows us to ensure that  $\xi_N$ is the unique solution to the following equation:
$$\frac{q_0}{\gamma a} e^{-z^2}=z,$$
as was obtained in \cite{Ta81}.
\medskip

\begin{coro}
The solutions provided in  \cite{ZWB}, \cite{ZhXi}   can be recovered from Theorem  \ref{teo: P}  by taking $\lambda=0$, $\beta\in\mathbb{R}^+$ and $\delta=0$.
\end{coro}

\noindent Taking $\delta=0$, we get that $L=\gamma s^{\beta}(t)$, i.e.  a power function of the position.  For such a case, by taking into account that $\alpha=\beta$ we automatically obtain the solutions already presented in literature. It must be pointed out that if $\beta$ is an integer, properties (\ref{Prop-3})-(\ref{Prop-4}) should be applied.

\begin{obs}
It must be noticed that in case we want to recover a latent heat defined as  $L=\frac{\gamma}{\dot{s}(t)}$, we have to set $\beta=0$, $\delta=-1$, and thus $\alpha=1$. However we can not recover the solution given in \cite{ZBL}, due to the fact that the boundary condition imposed at the fixed face (Dirichlet) does not agree with the boundary condition considered  in problem  $(P)$.
\end{obs}

\begin{coro}
The solution to the classical Stefan problem with a Robin boundary condition at the fixed face can be recovered from Theorem  \ref{teo: P}  by taking $\lambda=1$, $\beta=\delta=0$ (See \cite{Ta17}).
\end{coro}

\begin{coro}
The solution to the  Stefan problem studied in \cite{BoTa-1} can be recovered from Theorem  \ref{teo: P}  by taking $\lambda=1$, $\beta\in \mathbb{R}^{+}$ and $\delta=0$.
\end{coro}


\section{Equivalence to the problem with Dirichlet condition}

In \cite{BoTa}, the unique similarity solution of a problem defined  by equations (\ref{EcCalor})-(\ref{CondInicialFrontera})  with a Dirichlet boundary condition at the fixed face  characterized by $u_0>0$ was obtained, i.e.:
\begin{equation}
u(0,t)=t^{\tfrac{\alpha}{2}}u_{0}>0,\qquad t>0. \label{CondTempFrontera}
\end{equation}
The problem defined by conditions (\ref{EcCalor})-(\ref{CondInicialFrontera}) and (\ref{CondTempFrontera}) will be referred to as \textbf{problem ($P_D$)} and its solution will be referred to as the pair $(u_D,s_D)$.

According to \cite{BoTa}, if $\beta$ and $\delta$ are arbitrary real constants satisfying $\beta\geq \max(\delta,-1-\delta)$, taking $\alpha=\beta-\delta$, the unique solution to problem ($P_D$) is given by:
\begin{align}
& u_D(x,t)=t^{\tfrac{\alpha}{2}}\left[C_{1D} M\left(-\tfrac{\alpha}{2},\tfrac{1}{2},-\eta^2 \right)+C_{2D} \eta M\left(-\tfrac{\alpha}{2}+\tfrac{1}{2},\tfrac{3}{2},-\eta^2 \right)  \right], \label{FormaGeneraluD}\\
& s_D(t)=2\xi_D a \sqrt{t},\label{FormaGeneralsD}
\end{align}
where 
$$C_{1D}=u_0, \qquad \qquad C_{2D}=\frac{-u_0 M\left( -\tfrac{\alpha}{2},\tfrac{1}{2},-\xi_D^2\right)}{\xi_D M\left( -\tfrac{\alpha}{2}+\tfrac{1}{2},\tfrac{3}{2},-\xi_D^2\right)},$$
and $\xi_D$ is the unique positive solution of the following equation
\begin{equation}
\frac{k u_0}{\gamma a^{\beta+\delta+2} 2^{\beta+1}} f(z)=z^{\beta+\delta+1},\qquad \qquad z>0,\label{EcxiTemp}
\end{equation}
with 
\begin{equation}
f(z)=\frac{1}{zM\left( \tfrac{\alpha}{2}+1,\tfrac{3}{2},z^2\right)}.\label{Funcion-f}
\end{equation}

In this section we will study conditions on the data of the problem ($P$)  that guarantee its equivalence with the problem $(P_D)$. For equivalence it will be understood that both problems have the same solution.

Consider the problem ($P$) with a  given    data $\lambda$, $u_\infty$, $h_0$ whose solution $(u,s)$ is given by formulas (\ref{FormaGeneralu}), and (\ref{FormaGenerals}) under the hypothesis  that $\beta$ and $\delta$ are arbitrary real constants with $\beta\geq \max(\delta,-1-\delta)$, and $\alpha=\beta-\delta$. Computing $u(0,t)$ it is obtained that
$$u(0,t)=\dfrac{u_\infty \xi M\left( -\tfrac{\alpha}{2}+\tfrac{1}{2},\tfrac{3}{2},-\xi^2\right)}{\tfrac{k}{2ah_0} M\left(-\tfrac{\alpha}{2},\tfrac{1}{2},-\xi^2 \right)+\lambda \xi M\left( -\tfrac{\alpha}{2}+\tfrac{1}{2},\tfrac{3}{2},-\xi^2\right)}t^{\alpha/2},$$
with $\xi$ defined as the unique positive solution to equation (\ref{EcXi}).
Suppose now that we fix $$u_0=\tfrac{u_\infty \xi M\left( -\tfrac{\alpha}{2}+\tfrac{1}{2},\tfrac{3}{2},-\xi^2\right)}{\tfrac{k}{2ah_0} M\left(-\tfrac{\alpha}{2},\tfrac{1}{2},-\xi^2 \right)+\lambda \xi M\left( -\tfrac{\alpha}{2}+\tfrac{1}{2},\tfrac{3}{2},-\xi^2\right)},$$ 
and we solve the problem ($P_D$) obtaining $(u_D,s_D)$. Notice that  the moving boundary $s_D$ is characterized by a dimensionless coefficient $\xi_D$ that will be the unique solution to equation (\ref{EcxiTemp}), i.e.
\begin{equation}\label{Relacion q0u0-1}
\dfrac{k}{\gamma a^{\beta+\delta+2}2^{\beta+1}}\frac{1}{zM\left( \tfrac{\alpha}{2}+1,\tfrac{3}{2},z^2\right)} \tfrac{u_\infty \xi M\left( -\tfrac{\alpha}{2}+\tfrac{1}{2},\tfrac{3}{2},-\xi^2\right)}{\left[\tfrac{k}{2ah_0} M\left(-\tfrac{\alpha}{2},\tfrac{1}{2},-\xi^2 \right)+\lambda \xi M\left( -\tfrac{\alpha}{2}+\tfrac{1}{2},\tfrac{3}{2},-\xi^2\right)\right]}=z^{\beta+\delta+1}.
\end{equation}

Notice that if we put $z=\xi$, the prior equation reduces to  equation (\ref{EcXi}), meaning that $z=\xi$ constitutes a solution to (\ref{Relacion q0u0-1}). Therefore, as the unique solution to (\ref{Relacion q0u0-1}) is given by $\xi_D$, it results that $\xi_D=\xi$.
Then, it follows easily that $C_{1D}=C_{1}$, $C_{2D}=C_{2}$ obtaining as consequence that the solution $(u_D, s_D)$ with the $u_0$ data given in function of $\lambda,u_\infty,h_0$, coincides with the solution $(u,s)$ of the problem ($P$).

Conversely, consider the problem ($P_D$) with a given data $u_0$ whose solution $(u_D,s_D)$ is given by formulas (\ref{FormaGeneraluD}) and (\ref{FormaGeneralsD}) under the assumption that $\beta$ and $\delta$ are arbitrary real constants, $\beta\geq \max(\delta,-1-\delta)$, and $\alpha=\beta-\delta$. Computing $\frac{\partial u_D}{\partial x}(0,t)$ it is obtained that:
$$\frac{\partial u_D}{\partial x}(0,t)=\frac{-t^{\tfrac{\alpha-1}{2}} u_0 M\left(-\tfrac{\alpha}{2},\tfrac{1}{2}-\xi_D^2\right)       }{2a  \xi_D M\left(-\tfrac{\alpha}{2}+\tfrac{1}{2},\tfrac{3}{2},-\xi_D^2\right)}.$$
Let us consider ($P$) with the data $h_0$ given by
$$h_0=-\dfrac{ k u_0 M\left(-\tfrac{\alpha}{2},\tfrac{1}{2},-\xi_D^2 \right)}{2a\xi_D M\left(-\tfrac{\alpha}{2}+\tfrac{1}{2},\tfrac{3}{2},-\xi_D^2 \right) \left(\lambda u_0-u_{\infty} \right)},$$
fixing $\lambda$ and $u_\infty$ such that $\lambda u_0<u_\infty$. The solution $(u,s)$ of this problem can be obtained by (\ref{FormaGeneralu}) and (\ref{FormaGenerals}). The free boundary  $s$ is characterized by a dimensionless coefficient $\xi$ that is the unique solution of equation (\ref{EcXi}), i.e. satisfies:
\begin{equation}\label{Relacion q0u0-2}
  \dfrac{k u_\infty}{\gamma 2^{\beta+1} a^{\beta+\delta+2}} \frac{1}{\left[  \tfrac{k}{2ah_0} M\left(\tfrac{\alpha}{2}+\tfrac{1}{2},\tfrac{1}{2},z^2 \right)+\lambda z M\left( \tfrac{\alpha}{2}+1,\tfrac{3}{2},z^2\right)  \right] }=z^{\beta+\delta+1},\qquad z>0.
\end{equation}
The prior equation has $z=\xi_D$ as a solution due to the fact that  if we replace $z$ by $\xi_D$, it is obtained that equation (\ref{Relacion q0u0-2}) is equivalent to equation (\ref{EcxiTemp}). As (\ref{Relacion q0u0-2}) has a unique solution given by $\xi$, we claim that $\xi=\xi_D$. In  addition, by some computations, it becomes $C_{1}=C_{1D}$, $C_{2}=C_{2D}$ and so the solution $(u, s)$ to problem ($P$) given by a data $h_0$ in function of $u_0$ is equal to the solution $(u_D,s_D)$ to problem ($P_D$). 
Therefore the following theorem holds

\begin{teo}\label{EquivP} If $\beta$ and $\delta$ are arbitrary real constants satisfying $\beta\geq \max(\delta,-1-\delta)$ and $\alpha=\beta-\delta$, then the problem ($P$) defined by condition (\ref{EcCalor})-(\ref{CondConvFronteraG}) is equivalent  to problem ($P_D$) defined by (\ref{EcCalor})-(\ref{CondInicialFrontera}) and (\ref{CondTempFrontera}), when the parameters $\lambda, u_\infty$ and $h_0$ in the problem ($P$) are related with the parameter $u_0$ in problem ($P_D$) by the following expression: 
\begin{equation}\label{RelacionFlujoTemperatura}
u_0=\tfrac{u_\infty \xi M\left( -\tfrac{\alpha}{2}+\tfrac{1}{2},\tfrac{3}{2},-\xi^2\right)}{\tfrac{k}{2ah_0} M\left(-\tfrac{\alpha}{2},\tfrac{1}{2},-\xi^2 \right)+\lambda \xi M\left( -\tfrac{\alpha}{2}+\tfrac{1}{2},\tfrac{3}{2},-\xi^2\right)}.
\end{equation}
The coefficient $\xi$ makes reference to the unique solution of equation (\ref{EcXi}) for problem ($P$) which will coincide with the unique solution of (\ref{EcxiTemp}) for problem ($P_D$).
\end{teo}

As a consequence of the above result, by fixing $\lambda=0$ and $\lambda=1$, respectively we can obtain the following corollaries:

\begin{coro}{\sc\textbf{(case $\lambda=0$)}} If $\beta$ and $\delta$ are arbitrary real constants satisfying $\beta\geq \max(\delta,-1-\delta)$ and $\alpha=\beta-\delta$, then the problem ($P_N$) defined by condition (\ref{EcCalor})-(\ref{CondInicialFrontera}) and (\ref{CondFlujoFrontera}) is equivalent  to problem ($P_D$) defined by (\ref{EcCalor})-(\ref{CondInicialFrontera}) and (\ref{CondTempFrontera}), when the parameter $q_0$ in the problem ($P_N$) is related with the parameter $u_0$ in problem ($P_D$) by the following expression: 
\begin{equation}\label{RelacionFlujoTemperatura}
u_0= \frac{2 a q_0}{k} \xi_N \frac{M\left( -\tfrac{\alpha}{2}+\tfrac{1}{2},\tfrac{3}{2},-\xi_N^2\right)}{M\left(-\tfrac{\alpha}{2},\tfrac{1}{2},-\xi_N^2 \right)}.
\end{equation}
The coefficient $\xi_N$ makes reference to the unique solution of equation (\ref{EcxiFlujo}) for problem ($P_N$) which will coincide with the unique solution of (\ref{EcxiTemp}) for problem ($P_D$).
\end{coro}

\begin{coro}{\sc\textbf{(case $\lambda=1$)}} If $\beta$ and $\delta$ are arbitrary real constants satisfying $\beta\geq \max(\delta,-1-\delta)$ and $\alpha=\beta-\delta$, then the problem ($P_R$) defined by condition (\ref{EcCalor})-(\ref{CondInicialFrontera}) and (\ref{CondConvFrontera}) is equivalent  to problem ($P_D$) defined by (\ref{EcCalor})-(\ref{CondInicialFrontera}) and (\ref{CondTempFrontera}), when the parameters $h_0$, $u_{\infty}$ in the problem ($P_R$) are related with the parameter $u_0$ in problem ($P_D$) by the following expression: 
\begin{equation}\label{RelacionConvTemperatura}
u_0=\frac{ u_{\infty} \xi_R M\left(-\tfrac{\alpha}{2}+\tfrac{1}{2},\tfrac{3}{2},-\xi_R^2 \right)}{\tfrac{k}{2a h_0} M\left(-\tfrac{\alpha}{2},\tfrac{1}{2},-\xi_R^2 \right)+\xi M\left(-\tfrac{\alpha}{2}+\tfrac{1}{2},-\xi_R^2 \right)}.
\end{equation}
The coefficient $\xi_R$ makes reference to the unique solution of equation (\ref{EcXiConv}) for problem ($P_R$) which will coincide with the unique solution of (\ref{EcxiTemp}) for problem ($P_D$).
\end{coro}

\section{Asymptotic behaviour when the coefficient $h_0\to\infty$}

In this subsection we are going to analyse the behaviour of the problem ($P_R$) when the  coefficient $h_0>0$ which characterizes the heat transfer coefficient   at the fixed face $x=0$  tends to infinity. Due to the fact that the solution of this problem depends on $h_0$, we will rename it. Thus, we will consider $u_{R}(x,t,h_0):=u_R(x,t)$ and $s_{R }(t):=s_R(t,h_0)$ defined by equations (\ref{FormaGeneralu-Conv})-(\ref{FormaGenerals-Conv}).

Let us define the problem ($P_{D\infty}$) defined by  conditions (\ref{EcCalor})-(\ref{CondInicialFrontera}) and the following condition of Dirichlet type at the fixed face $x=0$ given by
\begin{equation}
u(0,t)=t^{\tfrac{\alpha}{2}}u_{\infty}, \label{CondTempInfintaFrontera}
\end{equation}
where $u_{\infty}$ corresponds to the data of the problem ($P_R$). Notice that the solution $(u_{D\infty},s_{D\infty})$ to problem ($P_{D\infty}$) can be obtained from (\ref{FormaGeneraluD}) and (\ref{FormaGeneralsD}) replacing $u_0$ by $u_\infty$. 

Then we are going to state the following result:

\begin{teo}\label{TeoConvergence}
If $\beta$ and $\delta$ are arbitrary real constants satisfying $\beta\geq \max(\delta,-1-\delta)$, and $\alpha=\beta-\delta$
the problem ($P_R$) converges to problem ($P_{D\infty}$) when $h_0$ tends to infinity, i.e.:
\be
\lim\limits_{h_0\rightarrow +\infty} \text{P}_R=\text{P}_{D\infty}. 
\ee
In this context the term ``convergence" means that:

\begin{equation}
\left\lbrace 
\begin{array}{lll}
\lim\limits_{h_0\rightarrow +\infty} \xi_R(h_0)&=& \xi_{D\infty},\\
\lim\limits_{h_0\rightarrow +\infty} s_{R}(t,h_0)&=& s_{D\infty}(t),\qquad \forall t>0, \\
\lim\limits_{h_0\rightarrow +\infty} u_R(x,t,h_0)&=&  u_{D\infty}(x,t),\quad \forall t>0,\;\; x>0.
\end{array}
\right.
\end{equation}

\end{teo}

\begin{proof}
On the one hand, the free boundary solution to problem ($P_R$) is characterized by a dimensionless parameter $\xi_{R}(h_0)$ that is the unique solution to equation (\ref{EcXiConv}), i.e.
$$f_{1}(z,h_0)=\frac{z^{\beta+\delta+1}}{C_{\infty}},\qquad z>0,$$
where $C_{\infty}=\dfrac{k u_{\infty}}{\gamma a^{\beta+\delta+2} 2^{\beta+1}}$ and $f_{1}(z,h_0):=f_1(z)$ given by (\ref{f1}).
On the other hand, the moving boundary $s_{D\infty}$ is characterized by a dimensionless parameter $\xi_{D\infty}$ which will be defined as the unique solution of the equation (\ref{EcxiTemp}) replacing $u_0$ by $u_{\infty}$, i.e.
$$f(z)=\frac{z^{\beta+\delta+1}}{C_{\infty}},\qquad z>0,$$
where $f$ is defined by (\ref{Funcion-f}).

We are going to prove that when $h_0\to \infty$, the coefficient $\xi_{R}(h_0)$ converges to the coefficient $\xi_{D\infty}$.
We know that $\frac{z^{\beta+\delta+1}}{C_{\infty}}$ is a strictly increasing function that goes from 0 to $+\infty$ when $z$ increases from 0 to $+\infty$; $f$ is a strictly decreasing function  that goes from $+\infty$ to 0 and $f_1(z,h_0)$ is a strictly decreasing function in $z$ as well but decreases from $\frac{2h_0 a}{k}$ to 0 when $z$ goes from 0 to $+\infty$.
After some computations, it can be seen that:
$$f(z)-f_1(z,h_0)=\frac{k}{2a h_0} \frac{f(z) M\left(\tfrac{\alpha}{2}+\tfrac{1}{2},\tfrac{1}{2},z^2 \right)}{\left[\tfrac{1}{f(z)}+\tfrac{k}{2ah_0} M\left(\tfrac{\alpha}{2}+\tfrac{1}{2},\tfrac{1}{2},z^2 \right)    \right]}>0, \qquad z>0.$$
Therefore it can be concluded that $0<\xi_R(h_0)<\xi_{D\infty}$, for all $h_0>0$. In addition, when $h_0\to \infty$ it can be easily seen that $f_{1}(z,h_0)\to f(z)$ and so $\xi_{R}(h_0)\to \xi_{D\infty}$. Once  this equality has been proved, by taking the limit in the definitions of $C_{1R}$ and $C_{2R}$ one can obtain the required convergence for $u_{R}$ and $s_R$.

\end{proof}

\section{Computational examples}

In this section, we present and discuss some computational examples.

From Theorem \ref{teo: P}, the solution to problem $(P)$ is characterized by a dimensionless parameter $\xi$ defined as the unique positive solution to equation (\ref{EcXi}). This equation can be rewritten as $F_\lambda(z)=0$ with
\begin{equation}\label{NewtonGeneral}
F_\lambda(z)= \dfrac{k u_\infty}{\gamma 2^{\beta+1} a^{\beta+\delta+2}} f_\lambda(z)-z^{\beta+\delta+1}=0,\qquad z>0.
\end{equation}
To solve the nonlinear equation $F_\lambda(z)=0$ we apply the following Newton's iteration formula:
\begin{equation}
z_{i+1}=z_i-\dfrac{F_\lambda(z_i)}{F_\lambda'(z_i)},
\end{equation}
where $z_i$ is the value of $z$ at the $i$th iteration step and
\begin{equation}
F'_\lambda(z)=\tfrac{-k u_\infty}{\gamma 2^{\beta+1} a^{\beta+\delta+2}} \tfrac{\left[ \tfrac{k}{2ah_0} (\alpha+1) M\left( \tfrac{\alpha}{2}+\tfrac{3}{2},\tfrac{3}{2},z^2\right)+\lambda M\left( \tfrac{\alpha}{2}+1,\tfrac{1}{2},z^2\right)\right]}{\left[ \tfrac{k}{2ah_0} M\left(\tfrac{\alpha}{2}+\tfrac{1}{2},\tfrac{1}{2},z^2 \right)+\lambda \xi M\left( \tfrac{\alpha}{2}+1,\tfrac{3}{2},z^2\right)    \right]}   -(\beta+\delta+1)z^{\beta+\delta}.
\end{equation}

We have implemented a MATLAB program to compute the dimensionless coefficient $\xi$ for different values of the parameters. The stopping criterion used is the boundedness of the absolute error \mbox{$\vert z_{i+1}-z_i\vert<10^{-10}$}.

In addition, given that the latent heat behaves as a function of the free front, we will plot $L$ in order to show how it changes in time. Observe that
\begin{equation}
L=\gamma s^{\beta}(t)\dot{s}^{\delta}(t)=\gamma \left(2\xi a \sqrt{t} \right)^\beta \left(  \frac{\xi a}{\sqrt{t}}\right)^\delta= \gamma 2^\beta a^{\beta+\delta}\xi^{\beta+\delta} t^{\frac{\beta-\delta}{2}}.
\end{equation}
Therefore it is deduced that the latent heat behaves as a power of time i.e  $L\sim t^p$ with $p<1$ if $\beta-\delta<2$, $p=1$ for $\beta-\delta=2$ and $p>1$ in case $\beta-\delta>2$. It must be pointed out that in all cases $\beta$ should be $\beta\geq \max\lbrace \delta,-1-\delta\rbrace$ in order to meet the hypothesis of Theorem \ref{teo: P}.


Let us first analyse the problem with a Neumann boundary condition at the fixed face. From Corollary \ref{teo: Flujo}, the solution of the problem ($P_N$) is characterized by a dimensionless parameter $\xi_N$ defined as the unique solution of equation (\ref{EcxiFlujo}). This equation can be rewritten as  (\ref{NewtonGeneral}) specifying $\lambda=0$:
\begin{equation}
F_0(z)=\dfrac{Q}{2^\beta}g(z)-z^{\beta+\delta+1}=0,\qquad z>0,
\end{equation}
where $g$ is given by (\ref{f0}) and the dimensionless parameter $Q$ is defined by:
\begin{equation}
Q=\frac{q_0}{\gamma a^{\beta+\delta+1}}.
\end{equation}
 In Table 1 we present  the computational results of $\xi_N$ for different values of $Q$.

\begin{table}[h]
\begin{footnotesize}
\begin{center}
\caption{{\footnotesize Computational results of $\xi_N$}.}
\label{tab:1}       
\begin{tabular}{ll|cccccc}
\hline
								&& $Q=0.1$ \qquad & $Q=0.2$ \qquad  & $Q=0.3$\qquad & $Q=0.4$\qquad     & $Q=0.5$\qquad \\
								\hline
	$\delta=0$,& $\beta=0$		& 0.0990  &  0.1927  &  0.2777   & 0.3531  &  0.5237\\
			   & $\beta=1$		&0.2138   & 0.2912   & 0.3453    &0.3875   & 0.4225\\
\hline
	$\delta=-1/2$,& $\beta=0$	&0.0100   &  0.0398 & 0.0879      &  0.1496 & 0.2172\\
			   & $\beta=1$		& 0.1319  &  0.2016 & 0.2543      & 0.2970  & 0.2952\\
			   \hline
	$\delta=1$,& $\beta=1$		& 0.3534 & 0.4357   & 0.4904    &  0.5321  &  0.5661\\
			   & $\beta=3$		&0.3838  &  0.4323  & 0.4627    &   0.4851  &   0.5031\\
			   \hline
\end{tabular}
\end{center}
\end{footnotesize}
\end{table}

In Figure 1, we plot the coefficient $\xi_N$ that characterizes the free front $s_N$, for different values of the parameters $Q$, $\beta$, $\delta$.

In Figure 2, we can observe graphically  what  can be analytically deduced, in the sense that when $\delta=0$, $\beta=1$, we obtain that the latent heat behaves as power of time, i.e.  $L\sim t^p$ with $p=\frac{1}{2}<1$. In the case that $\delta=1$, $\beta=3$, we obtain that $p=1$ and for $\delta=1$, $\beta=4$, the power becomes $p=\frac{3}{2}>1$.

\begin{flushleft}
\begin{tabular}{cc}
\includegraphics[scale=0.217]{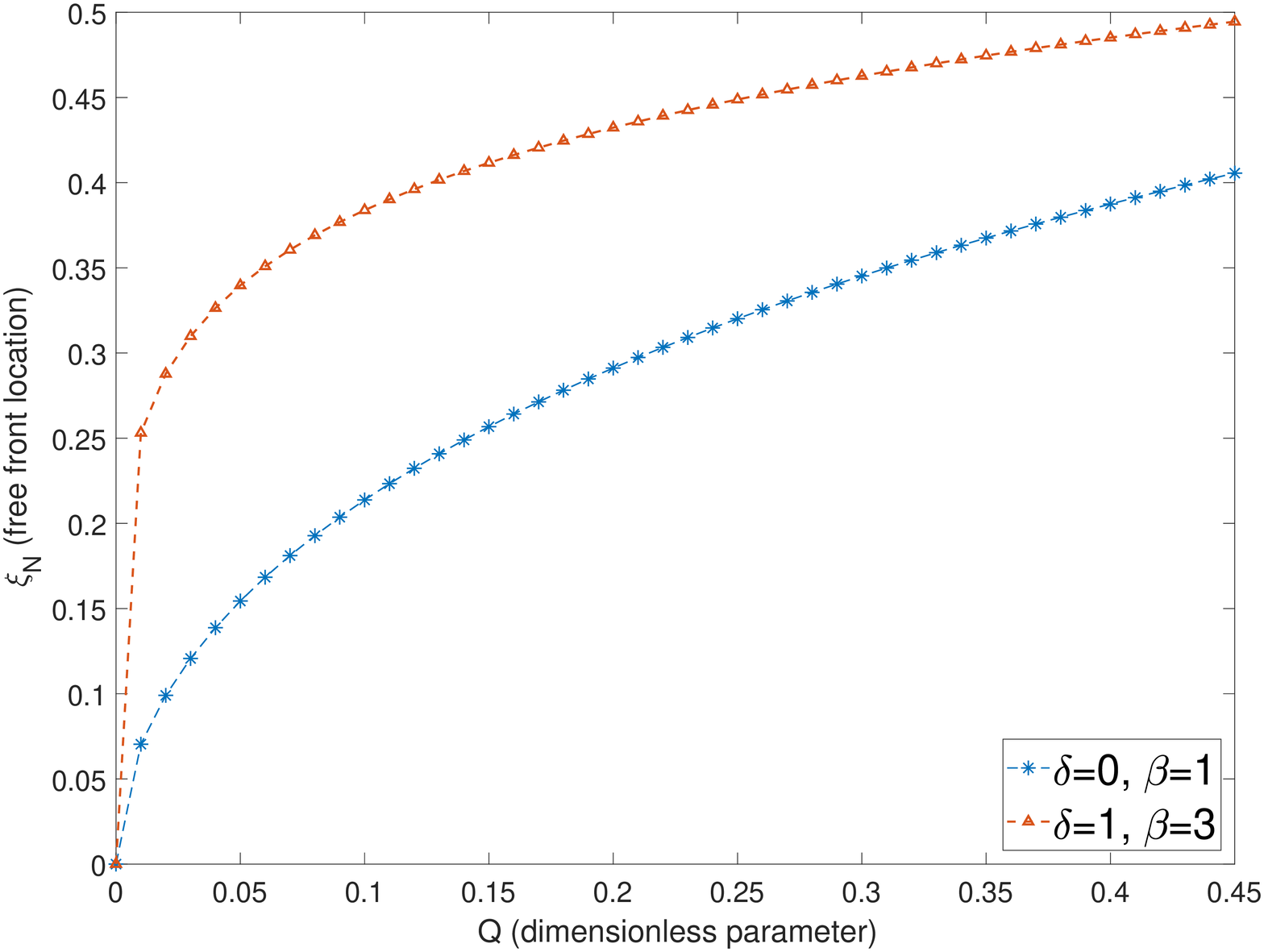} & \includegraphics[scale=0.217]{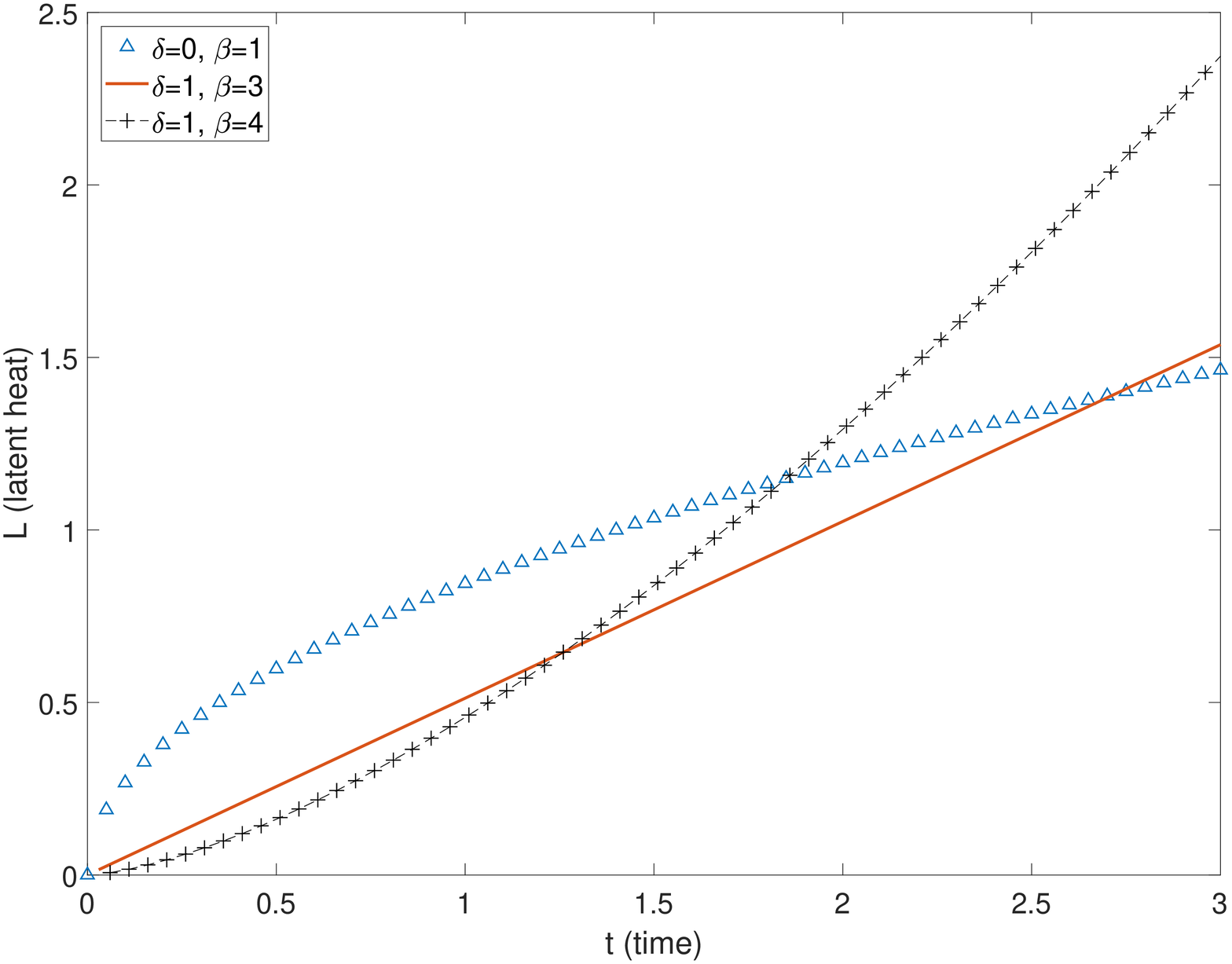}\\
{\scriptsize Figure 1: Plot of $\xi_N$ against  $Q$ for different values of $\beta$ and $\delta$.} & {\scriptsize Figure 2: Plot of $L$ against time for different values }\\
&  {\scriptsize of $\beta$ and $\delta$ assuming $Q=0.5$, $a=1$, $\gamma=1$.}
\end{tabular}
\end{flushleft}

Now, we turn to the problem with a Robin boundary condition at the fixed face. From Corollary \ref{teo: Conv}, the solution of the problem ($P_R$) is characterized by a dimensionless parameter $\xi_R$ defined as the unique solution of equation (\ref{EcXiConv}). This equation can be rewritten as (\ref{NewtonGeneral}) fixing $\lambda=1$:
\begin{equation}
F_1(z)=\dfrac{\text{Ste}}{2^{\beta+1}}f_{1}(z)-z^{\beta+\delta+1}=0,\qquad z>0,
\end{equation}
where $f_{1}$ is given by (\ref{f1}). Introducing the dimensionless  parameter Ste, which constitutes  a generalization of the Stefan number, and   the generalized Biot number
\begin{equation}
\text{Ste}=\dfrac{u_{\infty} k}{\gamma a^{\beta+\delta+2}},\qquad \qquad \text{Bi}=\frac{h_0 k}{a},
\end{equation}
we get that $f_{1}$ can be rewritten
$f_{1}(z)=\frac{1}{\left[\tfrac{1}{2\text{Bi}}M\left( \tfrac{\alpha}{2}+\tfrac{1}{2},\tfrac{1}{2}, z^2 \right)+z M\left( \tfrac{\alpha}{2}+1,\tfrac{3}{2},z^2\right)\right]}.$\\

In Table \ref{tab:2} we present  the computational results of $\xi_R$ for different values of Bi, $\beta$ and $\delta$, fixing Ste=0.5. Observe that the last column of the table intends to show that when Bi increases, $\xi_R$ becomes closer to $\xi_D$ which is the dimensionless parameter of the free front to the problem with Dirichlet condition. This convergence is in agreement with the prior section, taking into account that analysing $h_0\to\infty$ is equivalent to analysing $\text{Bi}\to\infty$.

\begin{table}[h!]
\begin{footnotesize}

\begin{center}
\caption{{\footnotesize Computational results of $\xi_R$}.}
\label{tab:2}       
\begin{tabular}{ll|cccc|c}
\hline
				&	Ste=0.5			& Bi=1 \qquad & Bi=10 \qquad  & Bi=50\qquad & Bi=100\qquad & $\xi_D$ \\
								\hline
	$\delta=0$,& $\beta=0$		&0.2926 &   0.4422  &  0.4601 &   0.4625  &  0.4648\\
			   & $\beta=1$		&0.3490 &   0.4485  &  0.4617  &  0.4635   & 0.4652\\
\hline
	$\delta=-1/2$,& $\beta=0$ & 0.1430  &  0.3375   & 0.3617   & 0.3648  &  0.3680\\
			   & $\beta=1$		& 0.2701 &   0.3837   & 0.3994  &  0.4015 &   0.4036\\
			   \hline
	$\delta=1$,& $\beta=1$		&0.4736 &   0.5514  &  0.5609    &0.5621 &   0.5634\\
			   & $\beta=3$		& 0.4615  &  0.5181 &   0.5260  &  0.5270 &   0.5281\\
			   \hline
\end{tabular}
\end{center}

\end{footnotesize}
\end{table}

In Figure 3,  we plot $\xi_R$ against Bi for different values of $\delta$ and $\beta$, fixing Ste=0.5.

\begin{center}
\includegraphics[scale=0.22]{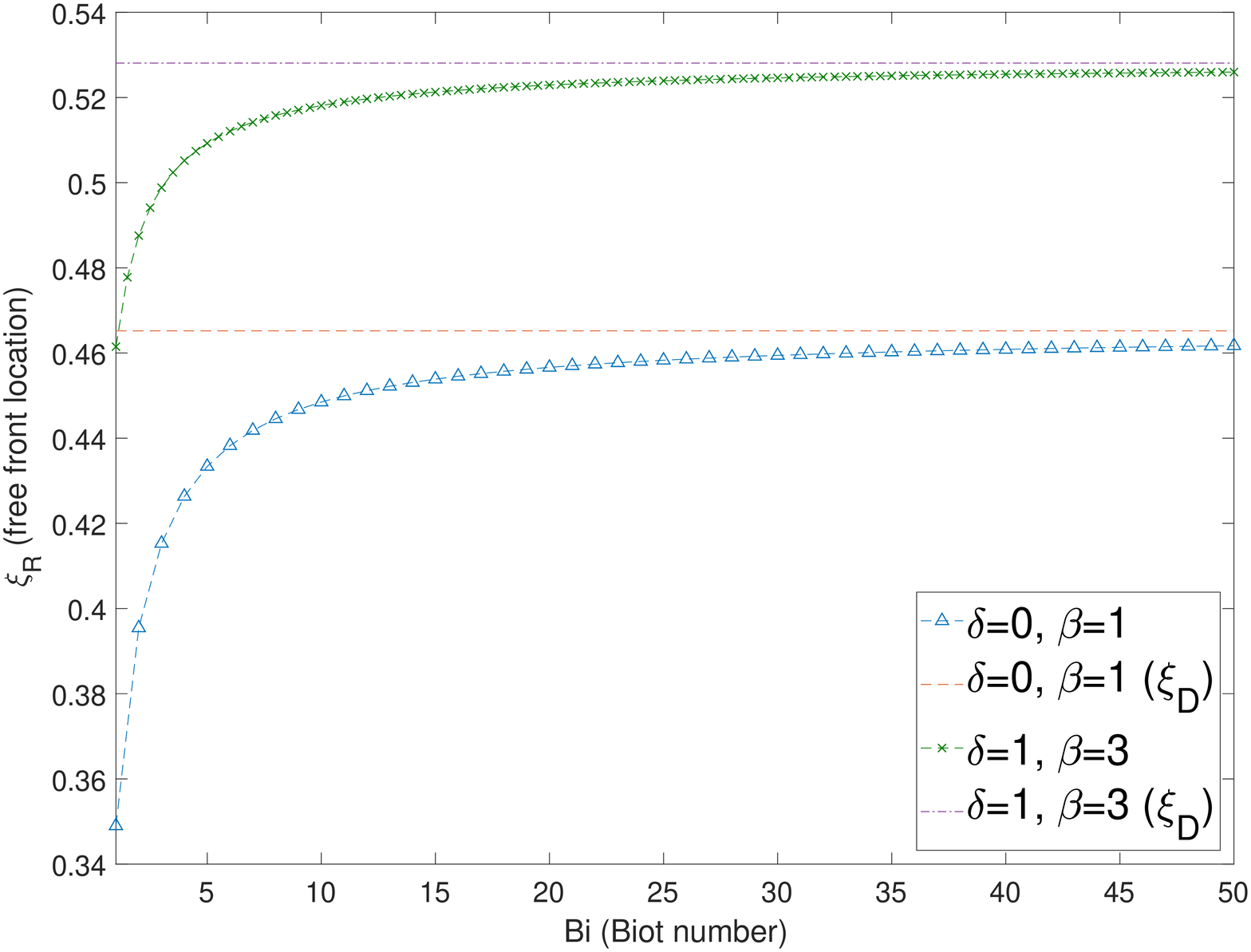}\\
{\footnotesize Figure 3: Plot of $\xi_R$ against  Bi for different values of $\beta$ and $\delta$, with Ste=0.5.}
\end{center}

\section{Conclusions}

In this paper two different one-phase Stefan-like problems were studied for a semi-infinite material. The main feature of both problems resides in the fact that a variable latent heat depending on the power of the position and the rate of change of the moving boundary is considered ($L=\gamma s^{\beta} \dot{s}^{\delta}$). Using Kummer functions, exact solutions of similarity type were obtained for the cases when a Neumann or Robin boundary conditions are imposed at the fixed face.

In addition, the necessary and sufficient relationships  between the data of the two problems in order to obtain an equivalence with the problem with a Dirichlet condition are obtained.

For the problem with Robin boundary condition,  the limit behaviour of the solution when the heat transfer coefficient at the fixed face goes to infinity was analysed, obtaining as a result, the convergence to the solution of a Stefan problem with a Dirichlet boundary condition.

This paper constitutes a mathematical generalization of the classical one  because it can be obtained by fixing the parameters $\beta=1$, $\delta=0$. Also, the results obtained when  a latent heat is considered  as a linear or power function of the position of the free boundary can be recovered.

We have provided tables and plots in order to show how the free front evolves in each case for specific values of the parameters.

It is worth to mention that finding exact solutions is meaningful not only to understand better the physical  processes involved  but also to
verify the accuracy of numerical methods that solve Stefan problems.

\section*{Acknowledgements}

We would like to thank three anonymous referees for their constructive comments which
improved the readability of the manuscript. The present work has been partially sponsored by the Projects PIP No 0534 from CONICET-UA and ANPCyT PICTO Austral 2016 No 0090, Rosario, Argentina.

\small{
}

\end{document}